\documentclass[10pt]{amsart}
\usepackage{a4wide,amsfonts,amsmath,amssymb,amsthm,epsfig,cite,graphicx,hyperref,
color, esint,fancyhdr, enumerate, latexsym,amsrefs}

\newtheorem{thm}{Theorem}[section]
\newtheorem{prop}[thm]{Proposition}
\newtheorem{lem}[thm]{Lemma}

\theoremstyle{definition}

\newtheorem{rem}[thm]{Remark}

\newcommand{\ep}{\epsilon}
\newcommand{\mbb}{\mathbb}

\newcommand{\pa}{\partial}
\newcommand{\mf}{\mathbf}

\newcommand{\Om}{\Omega}

\newcommand{\z}{\zeta}
\newcommand{\la}{\lambda}

\newcommand{\Ga}{\Gamma}
\newcommand{\ra}{\rightarrow}

\renewcommand{\Re}{\operatorname{Re}}

\hypersetup{
    colorlinks,
    citecolor=blue,
    filecolor=black,
    linkcolor=red,
    urlcolor=black
}

\begin{document}
\title{Remarks on the higher dimensional Suita conjecture}

\thanks{The second named author was partially supported by the DST-INSPIRE grant IFA-13 MA-21.}

\author{G.P. Balakumar, Diganta Borah, Prachi Mahajan and Kaushal Verma}

\address{G. P. Balakumar: Department of Mathematics, Indian Institute of Technology Palakkad, 678557, India}
\email{gpbalakumar@gmail.com}

\address{Diganta Borah: Indian Institute of Science Education and Research, Pune  411008, India}
\email{dborah@iiserpune.ac.in}

\address{Prachi Mahajan: Department of Mathematics, Indian Institute of Technology Bombay, Powai, Mumbai 400076, India}
\email{prachi@math.iitb.ac.in}

\address{Kaushal Verma: Department of Mathematics, Indian Institute of Science, Bangalore 560 012, India}
\email{kverma@iisc.ac.in}

\keywords{Suita conjecture, Bergman kernel, Kobayashi indicatrix}
\subjclass{32F45, 32A07, 32A25}

\begin{abstract}
To study the analog of Suita's conjecture for domains $D \subset \mbb C^n$, $n \ge 2$, B\l ocki introduced the invariant $F^k_D(z)=K_D(z)\la\big(I^k_D(z)\big)$, where $K_D(z)$ is the Bergman kernel of $D$ along the diagonal and $\la\big(I^k_D(z)\big)$ is the Lebesgue measure of the Kobayashi indicatrix at the point $z$. In this note, we study the behaviour of $F^k_D(z)$ (and other similar invariants using different metrics) on strongly pseudconvex domains and also compute its limiting behaviour explicitly at certain points of decoupled egg domains in $\mbb C^2$.
\end{abstract}

\maketitle

\section{Introduction}

\noindent
An essential step in B\l ocki's second proof of Suita's conjecture was to show that the Bergman kernel of a pseudoconvex domain in $\mathbb C^n$ admits a lower bound in terms of the volume of the sub-level sets of the pluricomplex Green's function. Let us recall some results from \cite{B1}. Write $\lambda$ for standard Lebesgue measure. For a domain $D \subset \mathbb C^n$, 
\[
K_D(z) = \sup \left\{ \vert f(z) \vert^2 : f \in \mathcal O(D), \int_D \vert f \vert^2 \; d \lambda \le 1 \right\}
\]
is the associated Bergman kernel (on the diagonal) and for $z \in D$,
\[
G_{D, z}(\zeta) = \sup \big\{ u (\zeta) : u \in PSH(D), u \le 0 \; \text{and} \;\limsup_{t \rightarrow z} \left( u(t) - \log \vert t - z \vert \right) < \infty \big\}
\]
is the pluricomplex Green's function for $D$ with pole at $z$. Theorem 1 of \cite{B1} shows that if $D$ is pseudoconvex, then for $a \ge 0$,
\[
K_D(z) \ge e^{-2na} \left( \lambda \big\{ \zeta : G_{D, z}(\zeta) < -a \big\}  \right)^{-1}
\]
for $z \in D$. When $D \subset \mathbb C$, the expression on the right approaches $ c^2_D(z)/ \pi$ as $ a \rightarrow \infty $, where 
$c_D(z)$ is the logarithmic capacity of the complement of $D$ with respect to the base point $z$. This shows that $c^2_D \le \pi K_D$ and 
this was precisely Suita's conjecture. Asking for the existence of this limit for a given $D \subset \mathbb C^n$ can then be regarded 
as an analog of Suita's conjecture in higher dimensions. To this end, note that if $D$ is convex, Lempert's theory shows that there 
is a diffeomorphism between $I^k_D(z)$, the Kobayashi indicatrix at $z$, and $D$. This observation is the basis of Theorem 2 of \cite{B1} which shows that
\[
\lim_{a \rightarrow \infty} e^{2na} \lambda \big \{ z : G_{D, z}(\zeta) < -a \big\}=  \la\big(I^k_D(z)\big)
\]
and hence 
\[
K_D(z) \la\big(I^k_D(z)\big) \ge 1
\]
everywhere on a convex domain $D$. This theme was pursued further in \cite{BZ}, \cite{B2} which shows that the limit exists if $D$ is pseudoconvex, the limiting value being $\la\big(I^a_D(z)\big)$, where 
\[
I^a_D(z) = \big\{ v \in \mbb C^n: \limsup_{\zeta \rightarrow 0} \big( G_{D, z}(z + \zeta v) - \log \vert \zeta \vert\big) < 0 \big\}
\]
is the indicatrix of the Azukawa metric
\[
a_D(z, v) = \limsup_{\zeta \ra 0} \vert \zeta \vert^{-1} \exp\big(G_{D, z}(z + \zeta v) \big).
\]

\medskip

For an invariant pseudo-metric $\tau$ on $D$, let $I^{\tau}_D(z)$ be the associated indicatrix at $z \in D$. The purpose of this note is to study the biholomorphic invariants 
\[
F^{\tau}_D(z) = K_D(z) \la\big(I^{\tau}_D(z)\big)
\]
when $\tau= c, a, k$ which are the Carath\'{e}odory, Azukawa and Kobayashi metrics respectively on a strongly pseudoconvex $D \subset \mbb C^n$. The prime motivation for doing so is the result presented in \cite{BZ} namely, 
\[
1 \le F^k_D(z) \le C^n
\]
where $C = 16$ or $4$ accordingly as $D$ is $\mathbb C$-convex or convex respectively. 

\begin{thm}\label{spscvx}
Let $D \subset \mbb C^n$ be a bounded $C^2$-smooth strongly pseudoconvex domain. Then for $\tau = c, a$ and $k$, 
\[
\lim_{z \to p_0} F^{\tau}_D(z)=1
\]
for all $p_0 \in \pa D$.
\end{thm}

The scaling method is useful here in that it allows us to reduce the problem to the case of the 
unit ball $\mbb B^n \subset \mbb C^n$ for which this invariant is equal to $1$ everywhere. Two ingredients 
are needed to do this. One, we prove a Ramadanov-type convergence theorem to control the sequence of scaled 
Bergman kernels, and this may perhaps be useful in other applications. Second, since $c \le a \le k$ on all domains, it follows that 
\[
I^k_D(z) \subseteq I^a_D(z) \subseteq I^c_D(z)
\]
and hence it suffices to prove a stability statement for the indicatrices of only the Carath\'{e}odory 
and Kobayashi metrics on the scaled domains. Since these metrics have the same boundary asymptotics
on strongly pseudoconvex domains, the volumes of their indicatrices have the same limiting values. In fact, we remark that the Theorem 1.1 holds for any invariant pseudometric $\tau$ satisfying $c \leq \tau \leq k$ by the same reasoning.
\medskip

\noindent We conclude the article by supplementing the observations in \cite{BZ}, \cite{BZ1} by analyzing  the boundary 
behaviour of $F^k_D$ for some decoupled egg domains which may possibly be non-convex. The main 
ingredients are the Wu metric and associated geometric analysis from \cite{CK2}, in estimating the measure of the indicatrix $I^k_D(z)$.

\section{Proof of Theorem \ref{spscvx}}
\noindent Let $ p^j $ be a sequence of points in $ D $ converging to $ p^0 \in \partial D $. Denote by $ \zeta^j \in \partial D $ the point closest to $ p^j $. Note that
$ \zeta^j \rightarrow p^0 $. Here and in the sequel, $ z \in \mathbb{C}^n $ is written as $ z = ('z, z_n) \in \mathbb{C}^{n-1} \times \mathbb{C} $. By 
\cite{Pinchuk-1980}, there exists a sequence $ \{ \phi^j \} $ of automorphisms of $ \mathbb{C}^n $ such that $ \phi^j ( \zeta^j) = ('0,0) $ for each $ j$ and 
the domains $ \phi^j(D) $ near the origin are defined by
\[
\{ z =('z, z_n) \in \mathbb{C}^n: 2 \Re \left(z_n + Q^j(z) \right) + H^j(z) + o ( |z|^2 ) < 0 \},
\]
where $ Q^j(z) = \sum_{\mu, \nu=1}^n q_{\mu \nu} (\zeta^j) z^{\mu} z^{\nu} $, $ H^j(z) = \sum_{\mu, \nu=1}^n h_{\mu \nu} (\zeta^j) z^{\mu} \overline{z}^{\nu} $ with
$ Q^j('z, 0) \equiv 0 $ and $ H^j('z, 0) \equiv |'z|^2 $. 

\medskip

\noindent Set $ \delta_j = \mbox{dist} \left( \phi^j(p^j), \partial \left(\phi^j(D) \right) \right) $, the Euclidean distance of 
$ \phi^j(p^j) $ to $ \partial \left(\phi^j(D)\right) $. Then $ \phi^j(p^j) = ('0, -\delta_j) $. Consider the dilations $ T^j : \mathbb{C}^n \rightarrow \mathbb{C}^n $ 
defined by 
\[
 T^j( 'z, z_n) = \left(   \frac{'z}{\sqrt{\delta_j}}, \frac{z_n}{{\delta}_j } \right).
\]
It follows that $ T^j \circ \phi^j (p^j) = ('0, -1) $ and the scaled domains $ D^j= T^j \circ \phi^j(D) $ converge in the local Hausdorff sense to the
unbounded realization of the unit ball, namely to
\[
 D_{\infty}= \{ z \in \mathbb{C}^n : 2 \Re z_n + |'z|^2 < 0 \}.
\]
Write $ ('0,-1) = p^* $ for brevity. Recall that the Cayley transform
\begin{equation} \label{E1}
 \Psi: ('z,z_n) \mapsto \left( \frac{{\sqrt{2}\;'z}}{1-z_n}, \frac{1+ z_n}{1 - z_n}\right)
\end{equation}
yields a biholomorphism from $ D_{\infty} $ onto $ \mathbb{B}^n$ that sends $p^*$ to the origin.

\medskip 

\noindent 
Observe that
\begin{equation}\label{F-scaling}
F^{\tau}_D(p^j)=F^{\tau}_{D^j} \left(p^*\right) = K_{D^j}( p^*) \la\big(I^{\tau}_{D^j} (p^*)\big).
\end{equation}

\noindent At this stage, we need (i) a Ramadanov type theorem for stability of the Bergman kernels, and (ii) stability of the Carath\'{e}odory and Kobayashi indicatrices 
under scaling.

\subsection{Stability of the Bergman kernels}
Recall that the Bergman kernel $K_{\Om}(z,w)$ of a domain $\Om \subset \mf{C}^n$ is the reproducing kernel for
the space $ A^2(\Omega) $ of square-integrable holomorphic functions in $ \Omega $, that is, $ f(z) = \int_{\Omega} f(w) K_{\Omega}(z,w) d \lambda(w) $ 
when $ f \in A^2(\Omega) $. Moreover, if $K_{\Om}(w)=K_{\Om}(w,w) > 0$, then $K_{\Om}(\cdot,w)/K_{\Om}(w)$ 
is the only function in $ A^2(\Omega) $ solving the extremal problem
\[
\text{minimize $\int_{\Om} \vert f \vert^2 \; d \lambda$ \; subject to $f \in A^2(\Om)$ and $f(w)=1$}.
\]

\begin{lem}\label{ramadanov}
Let $ \Omega^j $ be a sequence of domains in $ \mathbb{C}^n $ converging to $ \Omega \subset \mathbb{C}^n $ in the following way:
\begin{enumerate}[(i)] 
\item if $ S $ is a compact subset of $ \Omega $, then $ S \subset \Omega^j $ for all $ j $ large, and
 \item there exists a common interior point $ q $ of $ \Omega $ and $ \Omega ^j $ for all $ j $, such that for every $\ep>0$ there exists $j_{\ep} $  
 satisfying
\begin{equation} \label{E6}
\Omega^{j}-q \subset (1+\ep)(\Omega-q), 
\end{equation}
\end{enumerate}
for all $j \ge j_{\ep}$. Here, $ \Omega - q $ denotes the affine translation of $\Omega$ by $ - q $ and for $r > 0$, $r(\Omega - q)$ is the image of $\Omega - q$ under the homothety $T(v) = r(v - q) + q$ for $v \in \mathbb C^n$. Furthermore, assume that $ \Omega $ is 
star convex with respect to the point $ q $ and $ K_{\Omega} $ is non-vanishing along the diagonal. Then $K_{\Omega^j} \to K_{\Omega}$ uniformly 
on compact subsets of $ \Omega \times \Omega$.
\end{lem}

\begin{proof}
Without loss of generality we may assume that the point $q$ is the origin. 
If $\Om_0$ is a relatively compact sub-domain of $\Omega$, then $\Om_0 \subset \Omega^j$ for all $ j $ large. It follows that 
\begin{equation} \label{E5}
K_{\Omega^j}(z) \leq K_{\Omega_0}(z)
\end{equation}
for $ z \in \Om_0$ and for all $ j $ large. Also, recall that
\begin{equation} \label{E4}
\vert K_{\Omega^j}(z,w)\vert \leq \sqrt{K_{\Omega^j}(z)} \sqrt{K_{\Omega^j}(w)}
\end{equation}
for all $ z, w \in \Omega^j $ and for each $ j $. It follows from \eqref{E5} and \eqref{E4} that the sequence $\{K_{\Omega^j}\}$ is locally uniformly 
bounded on $ \Omega \times \Omega$, which provides a subsequence that converges locally uniformly to a function, say, $K_{\infty} $ on $ \Omega \times \Omega$. 

\medskip

\noindent The final step is to show that $K_{\infty}=K_{\Om}$ using the unique minimizing property of the Bergman kernel. To achieve this, fix $w\in \Om$ and 
note that $\Om^j \subset 2 \Omega $ for all large $j $. It follows that
$$ 
K_{\Omega^j}(w) \geq K_{2\Om}(w) > 0.
$$
for all $ j $ large, which in turn implies that 
\[
 K_{\infty}(w) \geq K_{2 \Omega}(w) > 0.
\]
Now let $f\in A^2(\Om)$ with $f(w)=1$. By Fatou's lemma, we obtain
\begin{alignat*}{3}
\int_{\Om_0}  \left\vert \frac{K_{\infty}(z,w)}{K_{\infty}(w)} \right\vert^2 \, d \lambda(z) & 
\leq \liminf_{j\to \infty} \int_{\Om_0} \left\vert \frac{K_{\Om^j}(z,w)}{K_{\Om^j}(w)}\right\vert^2 \, d \lambda(z) \;
 & \leq \liminf_{j \to \infty} \int_{\Om^j} \left\vert \frac{K_{\Om^j}(z,w)}{K_{\Om^j}(w)}\right\vert^2 \, d \lambda(z).
\end{alignat*}
Moreover, $ K_{\Omega^j} ( \cdot, z) $ reproduces the functions of $ A^2 (\Omega^j) $ and hence
\begin{alignat*}{3}
 \int_{\Om^j} \left\vert \frac{K_{\Om^j}(z,w)}{K_{\Om^j}(w)}\right\vert^2 \, d \lambda(z) =\frac{1}{K_{\Om^j}(w)}.
\end{alignat*}
Also, it follows from \eqref{E6} that
\begin{alignat*}{3}
 K_{(1+1/j)\Om}(w) \leq K_{\Om^j}(w)
\end{alignat*}
for each $ j $ and hence
\begin{alignat}{3} \label{int-K}
\int_{\Om_0}  \left\vert \frac{K_{\infty}(z,w)}{K_{\infty}(w)} \right\vert^2 \, d \lambda(z) \leq \liminf_{j \to \infty} \frac{1}{K_{(1+1/j)\Om}(w)}.
\end{alignat}
To find an upper bound for the right hand side above, set
\[
g_j(z)=\frac{f\left(\frac{z}{1+1/j}\right)}{f\left(\frac{w}{1+1/j}\right)},
\]
for $z \in (1+1/j)\Om$. Note that $f\left(\frac{w}{1+1/j}\right)\neq 0$ by the continuity of $f$ and hence $g_j$ is well-defined for all $ j $ large. Also, 
$g_j(w)=1$. Therefore,
\[
\frac{1}{K_{(1+1/j)\Om}(w)} \leq \int_{(1+1/j)\Om} \vert g_j(z) \vert^2 \; d \lambda(z) = 
\frac{(1+1/j)^{2n}}{\left\vert f\left(\frac{w}{1+1/j}\right)\right\vert^2}  \int_{\Om} \vert f(\z)\vert^2 d\lambda(\z).
\]
Combining the above observation with \eqref{int-K}, it follows that
\[
\int_{\Om_0}  \left\vert \frac{K_{\infty}(z,w)}{K_{\infty}(w)} \right\vert^2 \, d \lambda(z) \leq  \int_{\Om} \vert f(\z)\vert^2 d\lambda(\z).
\]
Since $\Om_0$ is an arbitrary compact subset of $ \Omega $, we obtain 
\[
\int_\Om \left\vert \frac{K_{\infty}(z,w)}{K_{\infty}(w)} \right\vert^2 \, d \lambda (z)  \leq \int_{\Om} \vert f(\z)\vert^2 d \lambda(\z),
\]
for every $ f \in A^2(\Omega) $ with $ f(w)=1 $. It follows from the minimizing property of the Bergman kernel function 
that $K_{\infty}(z,w)=K_{\Omega}(z,w)$. The above argument also shows that any convergent subsequence of $K_{\Omega^j}$ has limit $K_\Om$ and 
hence $K_{\Omega^j}$ itself converges to $K_\Om$.
\end{proof}

\noindent Another version of Ramanadov-type convergence theorem is stated below.
\begin{lem}\label{ramadanov-2}
Let $\{\Om^j\}$ be a sequence of domains in $\mathbb{C}^n$ that converges to a domain $\Om \subset \mathbb{C}^n$ in the following way:
\begin{enumerate}[(i)]
\item if $ S $ is a compact subset of $ \Omega $, then $ S \subset \Omega^j $ for all $ j $ large, and
\item there is a unit vector $v$ such that for every $\ep>0$, the translate $\Om+\ep v$ contains $\Om$ and also $\Om^{j}$ for $j$ large.
\end{enumerate}
Assume further that $K_\Om$ is non-vanishing along the diagonal. Then $K_{\Om^j} \to K_{\Om}$ uniformly on compact subsets of $\Om \times \Om$.
\end{lem}
\begin{proof}
Since the proof is exactly similar to previous case, we only outline the necessary modifications. Proceed as above, while working
with $\Om+v/j$ instead of $(1+1/j)\Om$. The analogous function $g_j$ on $\Om+v/j$ will be defined as
\[
g_j(z)=\frac{f\left(z-v/j\right)}{f\left(w-v/j\right)}
\]
and observe that
\[
\int_{\Om+v/j} \vert g(z) \vert^2 d \lambda(z)=\frac{1}{\left\vert f\left(w-v/j\right)\right\vert^2}  \int_{\Om} \vert f(\z)\vert^2 d \lambda(\z)
\]
to conclude.
\end{proof}

\subsection{Stability of the Carath\'{e}odory and Kobayashi indicatrices} Scale the strongly pseudoconvex domain $ D $ with respect to the 
base point $ p^0 \in \partial D $ and 
the sequence $ \zeta^j $. Let $ D^j $ and $ D_{\infty} $ be as described before. The first step towards establishing the stability of the indicatrices is to 
control $ k_{D^j}(\cdot, \cdot) $ and $ c_{D^j}( \cdot, \cdot) $ as $ j \rightarrow \infty $. 

\begin{lem} \label{met-conv}
For $ (z,v) \in D_{\infty} \times \mathbb{C}^n $, and for $ \tau = c, a $ and $ k$,
\begin{equation}\label{E2}
  \tau_{D^j}(z,v) \rightarrow \tau_{D_{\infty}} (z,v).
\end{equation}
Moreover, the convergence is uniform on compact sets of $ D_{\infty} \times \mathbb{C}^n $.
\end{lem}

\noindent The reader is referred to \cite{Seshadri&Verma-2006} for a proof when $ \tau = k $. To verify the above lemma for $ \tau = c $, one has to essentially 
repeat the arguments presented in \cite{Seshadri&Verma-2006} for the convergence of the Carath\'{e}odory distance on $ D^j $ and hence the proof is omited here; but 
note the key ingredients - firstly, if $ U $ is a 
sufficiently small neighbourhood of $ p^0 \in \partial D $, then
$ U \cap D $ is strictly convex and hence it follows from Lempert's work that $ c_{U \cap D} = k_{U \cap D} $. Secondly, the Carath\'{e}odory
metric can be localised near strongly pseudoconvex boundary points (see, for example, \cite{Graham}). 

\medskip 

\noindent Recall that 
\begin{equation*}
c_{D^j}(z,v) \leq a_{D^j}(z,v) \leq k_{D^j}(z,v)
\end{equation*}
for each $ j $ and $ c_{D^j}(z,v) \rightarrow c_{D_{\infty}}(z,v) $ and $ k_{D^j}(z,v) \rightarrow k_{D_{\infty}}(z,v) $. Moreover, since the limit 
domain $ D_{\infty} $ is biholomorphic to $ \mathbb{B}^n $, it follows that 
\[
c_{D_{\infty}}\equiv a_{D_{\infty}} \equiv k_{D_{\infty}}. 
\]
As a consequence, $ a_{D^j}(z,v) $ converges to $a_{D_{\infty}}(z,v)$. This completes the proof of Lemma \ref{met-conv}.

\medskip

\noindent The next step is a stability statement for the indicatrices of the scaled domains. As remarked earlier, it suffices to prove
the following lemma for $ \tau = c $ and $ k $. Here, we provide a proof for $ \tau = k $. It can be checked that 
the proof given below applies verbatim in the case $ \tau = c $. 

\begin{lem} \label{ind-conv} 
Let $ \tau = c, a $ and $ k$. Then for $z$ in any compact subset $S$ of $D_{\infty} $,
\begin{enumerate}
\item [(i)] $ I^{\tau}_{D^j} (z) $ is uniformly compactly contained in $ \mathbb{C}^n $ for all $ j $ large,
\item [(ii)] the indicatrices $ I^{\tau}_{D^j}(z) $ converge uniformly in the Hausdorff sense to $ I^{\tau}_{D_{\infty}}(z)$,
\end{enumerate}
and for each $z \in D$,
\begin{enumerate}
\item [(iii)]  the functions $\la\big(I^{\tau}_{D^j}(z)\big)$ converge to $\la\big(I^{\tau}_{D_{\infty}}(z)\big)$.
\end{enumerate}
\end{lem}

\begin{proof} To establish (i), recall that for $ v \in \mathbb{C}^n $, 
\begin{equation}\label{kob-limit}
 k_{D_{\infty}}(z,v) = k_{\mathbb{B}^n} \left( \Psi(z), d \Psi(z) v \right) = \left( \frac{|d\Psi(z) v|^2}{1- |\Psi(z)|^2} + 
 \frac{| \langle \Psi(z), d \Psi(z) v \rangle |}{\left(1- |\Psi(z)|^2\right)^2} \right)^{1/2}, 
\end{equation}
where $ \Psi $ is as described by \eqref{E1} and $ \langle \cdot, \cdot \rangle $ denotes the standard Hermitian inner product in $ \mathbb{C}^n $.
It follows that there is a uniform positive constant  $ C $ (depending only on $ S $) such that
\[
k_{D_\infty}(z,v) \geq C \vert v\vert
\]
for all $v \in \mf{C}^n$.

\medskip

\noindent Applying Lemma \ref{met-conv}, we see that 
\[
 \left| k_{D_\infty} (z,v) - k_{D^j}(z,v) \right| < \frac{C}{2}  \vert v \vert,
\]
and therefore
\begin{equation} \label{E3}
k_{D^j}(z,v) \geq \frac{ C }{2} \vert v\vert
\end{equation}
for all $z \in S$  and vectors $ v \in \mathbb{C}^n $ with $\vert v \vert=1$ and for all $ j $ large. Since $ k_{D^j}(z, \cdot) $ is homogeneous, the inequality \eqref{E3} holds for all $v \in \mathbb{C}^n$
and for all $ j $ large. A consequence of all of this is that, for all $ j $ large, the indicatrices $I^k_{D^j}(z)$ are contained in 
$ B \left(0, 2/C \right) \subset \mathbb{C}^n $, the Euclidean ball
centred at the origin and radius $ 2/C $, as required.
 
\medskip
 
\noindent To prove (ii), fix a ball $B= B(0,R) \subset \mathbb{C}^n $ containing $ I^{k}_{D^j} (z) $ for all $z \in S$ and $ j $ large, existence of which is guaranteed by (i). Let $\ep>0$. Suppose that $ v \in B $ and 
$ k_{D_{\infty}} (z,v) < 1-\ep/R$ (respectively $ > 1 + \ep/R$). Since $ k_{D^j}(z, v) \to k_{D_{\infty}} (z,v) $ 
uniformly on $ S \times B $, it follows that $ k_{D^j}(z,v) < 1 - \ep/2R$ (respectively $ > 1 + \ep/2R $) for all $j $ large. In particular, it follows that 
\[
(1-\ep/R)I^k_{D_{\infty}}(z)\subset I^k_{D^j}(z) \subset (1+\ep/R) I^k_{D_{\infty}}(z) 
\]
for all $z \in S$ and $j $ large which implies that the Hausdorff distance between $I^k_{D_{\infty}}(z)$ and $I^k_{D_j}(z)$ is less than or equal to $\ep$.

\medskip

\noindent For (iii), denote by $ \chi_A $, the indicator function of a subset $ A $ of $ \mathbb{C}^n $. Observe that $ \chi_{I^k_{D^j}(z)} $ 
converge pointwise $ \lambda$-almost everywhere to $ \chi_{I^k_{D_{\infty}}(z)} $. Indeed, if $ k_{D_{\infty}}(z,v)<1 $ (respectively $>1$) for 
some fixed $ v \in \mathbb{C}^n $, then $ k_{D^j}(z,v) < 1$ (respectively $>1$) for all $ j $ large.

\medskip

\noindent Moreover, as a
consequence of \eqref{kob-limit}, the set $ \{ v \in \mathbb{C}^n : k_{D_{\infty}}(z,v)=1 \} $ has zero Lebesgue measure. 

\medskip

\noindent The proof of (iii) now follows from the dominated convergence theorem.
\end{proof}

\noindent To conclude the proof of Theorem \ref{spscvx}, note that by Lemmas \ref{ramadanov} and \ref{ind-conv},
\begin{equation*}
K_{D^j}(p^*) \to K_{D_{\infty}}(p^*) \quad \text{and} \quad \la \big(I^{\tau}_{D^j}(p^*)\big) \to 
\la\big(I^{\tau}_{D_{\infty}} (p^*)\big).
\end{equation*}
Combining the above observations with \eqref{F-scaling} yields
\begin{equation*}
 F^{\tau}_D (p^j) \rightarrow F^{\tau}_{D_{\infty}} (p^*).
\end{equation*}
But $ D_{\infty} $ is biholomorphic to $ \mathbb{B}^n $ and hence
\begin{equation*}
 F^{\tau}_{D_{\infty}} (p^*) =  F^{\tau}_{\mathbb{B}^n} \left( ('0,0) \right) = 1,  
\end{equation*}
so that $ F^{\tau}_D (p^j) \rightarrow 1 $ as $ j \rightarrow \infty $.

\section{Localisation result}

\noindent It should be noted that $ F^{\tau}_D $ can be localised much like the invariant metrics $ \tau = c, a $ and $ k $ near peak points.


\begin{prop}\label{localisation2}
Let $ D \subset \mathbb{C}^n $ be a $ C^2$-smooth strongly pseudoconvex bounded domain and let $ p^0 \in \partial D$. Then for a sufficiently small 
neighbourhood $ U $ of $ p^0 $, and for $ \tau= c, a $ and $k $, 
\[
 \lim_{U \cap D \ni z \rightarrow p^0} \frac{F^{\tau}_{U \cap D} (z)}{F^{\tau}_D(z)} = 1.
\]
\end{prop}

\noindent This is immediate from the localisation properties of the Bergman kernel (see, for example, \cite{Hor}), 
the Azukawa metric (\cite{Nik}), the Carath\'{e}odory and the Kobayashi metrics (cf. \cite{Graham}) 
respectively, and hence the proof is omitted here.

\section{Concluding remarks.}
\noindent We would like to conclude this article with remarks about the boundary behaviour of
$F_D$ for $D$ varying through the increasing scale of egg domains, 
\[
E_{2\mu}=\{(z,w) \in \mathbb{C}^2 \; : \; \vert z \vert^2 + \vert w \vert^{2 \mu} <1\},
\]
where $\mu$ varies over the set of all positive real numbers. On convex eggs, i.e., $E_{2 \mu}$ for
$\mu \geq 1/2$, explicit expressions for $F_{E_{2\mu}}$ were obtained by B\l ocki -- Zwonek in \cite{BZ1}. We will rephrase what is already known from \cite{BZ1} about the convex case from the viewpoint of boundary behaviour and try to extend and tie it up with the non-convex case, to the extent we can. 

\medskip

To begin with, egg domains form perhaps the simplest class on which one may perform a concrete case study. While all the eggs in this scale (with the parameter $\mu$ varying through the scale of \textit{all} positive reals) are complete Reinhardt domains of holomorphy,
the scale includes a wide range of variety in their boundary geometry, containing for instance, the simplest model domains for smoothly bounded pseudoconvex domains of finite type (which happens when $\mu \in \mathbb{N}$). For non-integer values of $\mu$, the order of smoothness of 
$\partial E_{2 \mu}$ is $C^{[2\mu]}$; in short, all possible degrees of boundary smoothness, is represented in the scale. In particular, let us spell out that the degree of smoothness of the boundary is less than $C^1$ as soon as $m<1/2$. Moreover,
as soon as the value of the parameter $\mu$ drops below the threshold value $1/2$, the domains $E_{2 \mu}$ cease to be convex. 

\medskip

A common feature shared by all the $E_{2 \mu}$ -- whether $\mu<1/2$ or not --  that they are all bounded domains with non-compact automorphism group. 
Recall the classical theorem of H. Cartan that for any bounded domain $D$ in $\mathbb{C}^n$
its automorphism group is a real Lie group which is non-compact if and only if one -- and hence every -- of its orbits is non-compact. Therefore, to study the behaviour of a biholomorphically invariant function on a bounded domain $D$ with non-compact automorphism group, such as $F_D$ (whose values remain invariant along orbits), it is enough to study its boundary limits. 

\medskip

\noindent The aforementioned general fact about the non-compactness of the orbits means that every orbit accumulates on the boundary. Where and how they accumulate depends on the domain under consideration. For eggs domains this is known and may be recalled as follows. Firstly, view any particular egg $E_{2 \mu}$ as a disjoint union of its orbits (of the action of ${\rm Aut}(E_{2 \mu})$ on $E_{2 \mu}$). It is possible to mark off a convenient representative point for each orbit, the simplest of which is needless to say, the origin. The orbit of the origin is given by the intersection with $E_{2 \mu}$, of the
 complex hyperplane $\{z \in \mathbb{C}^2 \; : \; z_2=0\}$; the closure of this orbit meets 
 $\partial E_{2 \mu}$ precisely along its non-strongly pseudoconvex points. The
choice of convenient points for other orbits is facilitated by the fact that 
orbits of points of the form $(0,p)$ as $p$ varies in the interval $[0,1)$ exhaust $E_{2\mu}$ and moreover, no pair of such points belong to the same orbit. This can be seen explicitly by working with the specific automorphisms of these egg domains. We shall therefore refer to these points as `representative points' for these eggs and denote the set of such points by $S$, i.e.,
\[
S=\big\{(0,p) \in \mathbb{C}^2 \; :\; p \in \mathbb{R} \text{ with } 0\leq p<1\big\},
\]
which is contained in $E_{2\mu}$ for any $\mu$. To study any invariant function on any of these eggs therefore, it suffices to restrict attention to this segment $S$. In particular,
as $F_{E_{2 \mu}}$ is biholomorphism/automorphism-invariant, it is constant along any of the orbits and the range of $F_{E_{2 \mu}}$ equals $F(S)$.
Finally, we recall that any non-strongly pseudoconvex boundary point $q \in \partial E_{2 \mu}$ is a boundary-orbit-accumulation point for all orbits. We conclude that: for any of the non-strongly pseudoconvex points $q$ in $\partial E_{2 \mu}$, any neighbourhood $U$ (in $\mathbb{C}^2$), of $q$ -- howsoever small -- captures all the values 
attained by $F(z)$ as $z$ varies throughout $E_{2 \mu}$:
\[
F(E_{2 \mu} \cap U) = F(E_{2 \mu})=F(S).
\]


In particular, one of the boundary limits gives the value $1$, as obtained when we approach $q$ along the orbit of the origin $Z$; indeed, $F_{E_{2 \mu}}$ is constant on $Z$ and equals $1$ as the Bergman kernel function at the origin of any complete Reinhardt domain $R$ equals the inverse of the volume of $R$ and the Kobayashi indicatrix at the origin for $R$ is a copy of $R$ itself.
To highlight the main point here, first note that while the boundary limit of $F(z)$ as $z$ approaches a strongly pseudoconvex point on the boundary a domain always exists and equals $1$
as is guaranteed by Theorem \ref{spscvx}, the boundary limit of $F$ certainly fails to exist at any of the non-strongly pseudoconvex points on the boundary of domains as simple as convex eggs.


\medskip

\noindent  Although we do not obtain precise explicit expressions for $F_{E_{2\mu}}$ as in \cite{BZ1}, which as mentioned therein is already complicated,
we would like to state some estimates, even if coarse, for the `low-regularity' cum non-convex case i.e., for $F_{E_{2\mu}}$ when $\mu<1/2$. Owing to the structure of the orbits in the egg domains described above and as in \cite{BZ1}, it suffices to estimate the invariant function along the thin segment $S$. However, even this can be complicated for although explicit expressions for the Kobayashi metric is known for all the aforementioned egg domains, such expressions get even more complicated in the non-convex case, involving implicitly defined parameters to unravel which, requires solving highly non-linear (non-polynomial) equations. To circumvent this, we shall use 
instead, a geometric analysis of the Kobayashi indicatrix in \cite{CK2}. 

\medskip

A drawback, however, is that our estimates are possibly good only in a small neighbourhood of boundary points, small enough to atleast avoid the origin. Indeed, as already mentioned $F_{E_{2 \mu}}$ is $1$ at the origin (and thereby on its orbit $Z$), and the focus of the estimates below is on 
$F_{E_{2 \mu}}(z)$ for $z$ varying in the complement of $Z$; particularly in a small neighbourhood of any one 
of the strongly pseudoconvex points (points from $\partial E_{2 \mu} \setminus \overline{Z}$) such as the point $(0,1)$. Taking limits as $p \to 1$, the upper and lower bounds both approach $1$ in accordance with our Theorem \ref{spscvx} and are therefore not too coarse; indeed, they give an idea of the rate of convergence of $F_{E_{2 \mu}}$ to $1$ as we approach the representative strongly pseudoconvex point $(0,1)$ in $\partial E_{2 \mu} $, through the inner normal.

\begin{prop}
For every positive $\mu <1/2$, the following upper and lower bounds hold for $F_{E_{2 \mu}}$ at points of the segment $S$ of representative points.
\begin{itemize}
\item [(a)] 
\begin{equation*}\label{ub}
F_{E_{2\mu}}(0,p) \leq \frac{1}{\mu}\left(\frac{1- p^{2\mu}}{1- p^2}\right) 
- \frac{1-\mu}{2\mu} (1- p^{2\mu}),
\end{equation*}
\item [(b)] 
\begin{equation*}\label{lb-noncvx}
F_{E_{2\mu}}(0,p) \geq \frac{p^{2-2\mu}}{2 \mu^3}\left( \frac{1-p^{2\mu}}{1-p^2}\right)^3 
\left(1+\mu+p^2 -\mu p^2\right).
\end{equation*}
\end{itemize}
\end{prop}
\begin{proof}
Borrowing the computations for the Bergman kernel done in \cite{dA94}, we write down its expression for our egg domains:
\begin{equation}\label{Bergker}
K_{E_{2\mu}}(0,p) = \frac{\mu-1}{\pi^2 \mu} \frac{1}{(1-p^2)^2} + \frac{2}{\pi^2 \mu} \frac{1}{(1-\vert p\vert^2)^3}.
\end{equation}
To estimate the other factor involved in $F$, namely the volume of the Kobayashi indicatrix, we recall the Wu 
ellipsoid which is the best fitting Euclidean ellipsoid containing the Kobayashi indicatrix. From the expression 
for the Wu metric for the egg domains obtained in \cite{CK2}, we 
immediately see that this Wu-ellipsoid is described by
\begin{equation*}
\left\{(v_1, v_2) \in \mf{C}^2: \frac{\vert v_1\vert^2}{1-\vert p\vert^{2\mu}}+\frac{\vert v_2\vert^2}{(1-\vert p\vert^2)^2}<1\right\}.
\end{equation*}
The volume of this ellipsoid is
\begin{align*}
\frac{\pi^2}{2}(1-\vert p\vert^2)^2(1-\vert p\vert^{2\mu}).
\end{align*}
Multiplying this together with the expression in (\ref{Bergker}), renders (a).

(b) Recall from \cite{CK2} in the terminology therein, that the square transform $I^s$ of 
the Kobayashi indicatrix in the absolute space of $\mathbb{C}^2$ i.e., the first quadrant 
of $\mathbb{R}^2$, has boundary formed by the join of two curves called the {\it upper} K-curve 
and the {\it lower} K-curve. The lower K-curve is in fact a straight line segment, while the upper 
K-curve is a curve which can be realized (implicitly) as the graph of a strictly convex function 
on some strict sub-interval of $[0,1]$ containing $0$. It follows therefore that the 
lower K-curve when extended to meet the other axis, it remains within the closure of $I^s$. Consequently, the 
Kobayashi indicatrix in the tangent space to $E_{2 \mu}$ as $(0,p)$ contains the  Euclidean ellipsoid
\[
\left\{ (v_1,v_2) \in \mathbb{C}^2 \; : \; \frac{\vert v_1 \vert^2}{1-p^{2 \mu}}
+ \frac{\mu^2p^{2\mu-2}}{(1-p^{2\mu})^2}\vert v_2 \vert^2 <1 \right\}
\]
whose volume comes out to be 
\[
\frac{\pi^2}{2} \frac{(1-p^{2 \mu})^3}{\mu^2 p^{2\mu-2}}.
\]
This when multiplied by the expression of the Bergman kernel in (\ref{Bergker}), gives the stated lower bound for $F_{E_{2 \mu}}$.
\end{proof}

\end{document}